%% file: badia_martin_nguyen_bddc_subobjects.tex
\documentclass[final,3p]{elsarticle}

\usepackage{amssymb}

\usepackage{amsmath}
\usepackage{amssymb}
\usepackage{amsfonts}
\usepackage{tikz}
\usetikzlibrary{plotmarks}
\usepackage{pdfsync}
\usepackage{comment}
\usepackage{hyperref}
\usepackage{color}
\usepackage{makeidx}         
\usepackage{enumerate}
\usepackage{booktabs}
\usepackage{lineno}

\def\x{{\boldsymbol{x}}}

\def\VbddcC{\tilde{\mathbb{V}}_{\rm c}}

\def\VbddcF{{\widetilde{\mathbb{V}}_{\rm f}}}
\def\VbddcF1{{\tilde{\mathbb{V}}_{\rm f}}}
\def\VbddcC1{{\tilde{\mathbb{V}}_{\rm c}}}

\def\R{{\mathbb R}}

\newtheorem{theorem}{Theorem}
\newtheorem{definition}[theorem]{Definition}

\newtheorem{remark}[theorem]{Remark}

\newproof{proof}{Proof}

\journal{Applied Mathematics Letters}

\def\fullTitle{Balancing domain decomposition by  
constraints associated with subobjects
}
\def\sbAuthor{Santiago Badia}
\def\sbmail{sbadia@cimne.upc.edu}
\def\amAuthor{Alberto F. Mart\'in}
\def\ammail{amartin@cimne.upc.edu}
\def\htnAuthor{Hieu Nguyen}
\def\htnmail{nguyentrunghieu14@dtu.edu.vn}
\def\dtuAddress{Institute of Research and Development, Duy Tan 
University, 3 Quang Trung, Danang, Vietnam.} 

\def\upcAddress{Universitat Polit\`ecnica de Catalunya, Jordi Girona 1-3, Edifici C1, 08034 Barcelona, Spain.}   
\def\cimneAddress{CIMNE – Centre Internacional de M\`etodes Num\`erics en 
Enginyeria, UPC, Esteve Terradas 5, 08860 
Castelldefels, Spain.
 }

\def\fundthanks{This work has been partially funded by the European 
Union under the ExaQUte project within the H2020 Framework Programme 
(grant agreement 800898). Santiago Badia gratefully acknowledges the 
support received from the Catalan Government 
through the ICREA Acad\`{e}mia Research Program. Hieu Nguyen would like to thank the financial 
support from  Vietnam National Foundation for Science and Technology Development 
(NAFOSTED) under Grant No.~101.99-2017.13. Finally,
the authors thankfully acknowledge the computer resources at MareNostrum
IV and the technical support provided by Barcelona Supercomputer Center.}

\def\hnr#1{{#1}}

\begin{document}
	
	\begin{frontmatter}
		
		
		\title{\fullTitle}
		\author[one,two]{\sbAuthor}
		\ead{\sbmail}
		
		\author[one,two]{\amAuthor}
		\ead{\ammail}
		\author[three]
		{\htnAuthor \corref{cor1}}
		\ead{\htnmail}
		\cortext[cor1]{Corresponding author}
		

		\address[one]{\upcAddress}
		\address[two]{\cimneAddress}
		\address[three]{\dtuAddress}
		
		\def\myKeywords{BDDC, $\log(H/h)$, parallel 
		solver, parallel preconditioner, adaptive mesh}
		
		\begin{abstract}
			A simple variant of the BDDC 
			preconditioner 
			in which constraints are imposed on a selected set of 
			subobjects (subdomain subedges, subfaces and vertices 
			between pairs of 
			subedges) 
			is presented. 
			We are able to show that the condition number of the  
			preconditioner is bounded by $C \big(1+\log (L/h)\big)^2$, 
			where $C$ is a constant, and $h$ and $L$ are the characteristic 
			sizes of the mesh and the subobjects, respectively. As $L$ 
			can be chosen almost freely, the condition number can  
			theoretically be as small as $O(1)$. We will discuss the 
			pros and cons of the preconditioner and its application to 
			heterogeneous problems. 
			Numerical results on supercomputers are provided. 
		\end{abstract}
		
		\begin{keyword}
			BDDC \sep FETI-DP  \sep optimal preconditioner 
			\sep parallel solver \sep heterogeneous problems 
			
			\MSC[2010] 65N55 \sep 65N22 \sep 65F08
		\end{keyword}
		
	\end{frontmatter}
	
	


\section{Introduction}\label{sec1}
It is well-known in the literature 
\cite{klawonn.et.al.02,mandel_convergence_2003,mandel_algebraic_2005, 
widlundToselli05,brenner_bddc_2007,BrennerScott08}
that the condition number 
of the BDDC method and the closely related FETI-DP method is 
bounded by 
$\kappa \leq C \big(1+\log (H/h)\big)^2$,
where $C$ is a constant, and $H$ and $h$ are the characteristic sizes of the 
subdomain and the mesh, resp.  Consequently, when the mesh and its 
partition 
are fixed, users can only vary the number of constraints to adjust the 
convergence rate.

Traditionally, these constraints are values at subdomain corners, or 
average 
values on subdomain edges or faces. In the perturbed formulation of 
BDDC \cite{BadiaNguyenDD23,BadiaNguyen_perturbed}, any combination of vertex, 
edge or face constraints can be used. Normally, these combinations provide 
enough options for users to choose their desired convergence rate. However, 
for difficult problems, such as the ones with heterogeneous 
diffusion coefficients, see, e.g., 
\cite{BadiaMartinNguyen_multi_material}, the 
convergence of the BDDC method \cite{dohrmann_preconditioner_2003} can 
be 
unsatisfied even if all possible constraints are used.

In this work, we propose to impose constraints on coarse objects 
(subdomain subedges and 
subfaces, and vertices between subedges) that are 
adapted to the local mesh size and to the user's desired rate of 
convergence. We refer to this new variant as BDDC-SO.  
If $L$ is the characteristic size of the objects, we can prove that the 
conditioned number $\kappa$ of the
BDDC-SO preconditioned operator is bounded by   
\[\kappa \leq C \big(1+\log (L/h)\big)^2.\]

As $L$ can be as small as $h$, there is no restriction in the rate of 
convergence of BDDC-SO. However, the trade-off is having to solve larger coarse 
problems. Therefore, the BDDC-SO can be seen as a standalone method for 
small to medium problems (tested up to 260 million unknowns on 8K 
processors), or as a proven,  simple approach to improve the rate of 
convergence. In addition, BDDC-SO is perfectly robust w.r.t. the variation of 
the diffusion coefficient when coarse objects are chosen such that the  
coefficient associated with each object is constant or varies mildly. 
BDDC-SO also has great potentials to 
be extended to multi-level.
\section{Model problem and abstract multispace BDDC 
framework}\label{sec_model}
Let $\Omega$ be a bounded domain in $\R^n,\; 
n=2, 3$. We consider the Poisson's equation with
homogeneous Dirichlet conditions.  
Its weak formulation reads as follows: find $u(x)\in H^{1}_0(
\Omega )$ such that
\begin{equation}\label{weakform}
  a(u,v)= \langle F, v \rangle, \quad \forall v(x) \in H^1_0( \Omega 
  ), 
\end{equation}
where 
$
a(u,v):=\int_{\Omega}  \nabla u\cdot \nabla v\, dx,\;
\langle F, v \rangle:=\int_{\Omega} f(x)\, v(x) dx.   
$

Let $\mathcal{T}$ be a shape-regular 
mesh of $\Omega$  with characteristic mesh size $h$,
and $U\subset H^{1}_0(\Omega)$ be the corresponding Finite Element (FE) 
space of 
continuous piecewise linear functions  
that vanish on $\partial \Omega$. Then the discrete problem associated 
with  
\eqref{weakform} is to find $u\in U$ such that 
\begin{equation} \label{discrete_form} a(u,v)=
  \langle F, v \rangle, \quad \forall v\in U.  
\end{equation}

We quote the following definition of the abstract multispace BDDC 
preconditioner from 
\cite{mandel_multispace_2008}.
\begin{definition}\label{def:abstractBDDC} Assume that $a(\cdot,\cdot)$ 
is defined and symmetric 
	semipositive definite on some larger space $W\supset U$ and there 
	exist 
	subspaces $V_k\subset W$ 
	and projections $Q_k: V_k\rightarrow U$, 
	$k=1,\dots,N$, such that $V_k$ are a-orthogonal, i.e., $V_k\perp V_l$ 
	for 
	$k\neq 
	l$, 
	$a(\cdot,\cdot)$ is positive definite on each $V_k$, $ 
	\widetilde{W}=\sum_{k=1}^{N}V_k\supset  U$ and
	\[\forall u\in U:\left[u=\sum_{k=1}^N v_k, \; v_k\in 
	V_k\right]\Rightarrow u=\sum_{k=1}^{N} Q_k v_k. \]
	Then the abstract multispace BDDC preconditioner $B:r\in 
	U^{\prime}\rightarrow U$ is defined by
	\[B:r\mapsto u=\sum_{k=1}^{N} Q_k v_k,\quad v_k\in V_k: 
	a(v_k,z_k)=\langle r,Q_kz_k\rangle, \quad \forall z_k \in V_k.\]
\end{definition}
\section{Balancing domain decomposition by constraints associated with 
subobjects}\label{sec:formulation}
We will present the formulation of our proposed BDDC preconditioner  
using the abstract BDDC framework introduced in Section 
\ref{sec_model}. First, we consider a partition 
$\Theta=\{\Omega_i\}_{i=1}^N $ of 
$\Omega$ 
into
 non-overlapping open subdomains. We further assume that every subdomain $
\Omega_i \in \Theta$ is a union of elements in
$\mathcal{T}$ and is connected.  Let $\Gamma$ be the interface of 
$\Theta$. For each $\Omega_i$, let $W_i$ be 
the space of FE functions which vanish on 
$\partial\Omega\cap\partial\Omega_i$.  We define $W=W_1\times \cdots 
W_N$ and note that functions in $W$ can be discontinuous across 
$\Gamma$. 

Let $a^{\Theta}(\cdot,\cdot)$ be the extension of 
$a(\cdot,\cdot)$ \hnr{on $W$ associated} with $\Theta$: 
$a^{\Theta}(u,v)=\sum_{\Omega_i\in \Theta} \int_{\Omega_i}\nabla u 
\nabla v dx$. Denote by 
$U_I$ the space of functions in $U$ that are zero on $\Gamma$. 
Let 
$P$ be the $a^{\Theta}$-orthogonal projection from $W$ onto $U_I$. The 
functions in $(I-P)W$ are called \emph{discrete harmonic}. 
They are 
$a^{\Theta}$-orthogonal with functions in $U_I$ and they possess the 
smallest 
energy among functions in $W$ having the same values on $\Gamma$. 

Given a mesh partition, the BDDC method is 
characterized by the selection of certain \emph{coarse degrees of 
freedom} and some \emph{weighting operator} that projects functions in 
$W$ onto $U$.
\input{fig1.tex}
Coarse degrees of 
freedom are usually defined as functionals, which are values at or 
averages over some geometrical coarse objects. In the standard BDDC method, 
these coarse objects are vertices and/or edges and/or faces of 
subdomains. 
For systematic ways of classifying degrees of 
freedom into objects, we refer the reader to 
\cite{klawonn07_heterogenous,BadiaMartinNguyen_multi_material}. 

In this 
work, we propose to further partition edges and faces of the subdomains 
into subedges, subfaces and vertices between pairs of subedges. One way 
to obtain 
this is to 
partition each subdomain $\Omega_i$ into smaller \hnr{subsubdomains}. 
These 
new 
\hnr{subsubdomains} together form a new partition 
$\widehat{\Theta}=\{\widehat{\Omega}_j\}_{j=1}^{\widehat{N}}$ of 
$\Omega$. Denote by $\widehat{\Gamma}$ the interface of 
$\widehat{\Omega}$, we emphasize that $\widehat{\Gamma}\supset\Gamma$. 
The BDDC-SO method selects  
objects from \emph{the set of   
subdomain vertices, edges and faces of  
$\widehat{\Theta}$} to define its coarse degrees of freedom. A simple 
example in 
2D is illustrated in Figure \ref{fig:objects}. \hnr{On the left, 
standard 
coarse objects consisting of a vertex and four edges   
are 
shown.} On the right, obtained from a finer partition 
$\{\widehat{\Omega}_j\}_{j=1}^{8}$, we have 8 subedges and 
5 vertices between them. We emphasize that \emph{one normally does not 
have to impose 
constraints 
on all available objects}.

\hnr{\begin{remark}
The decision to break each original subdomain $\Omega_i$ into smaller 
subsubdomains $\widehat{\Omega}_j$ can be made locally in practice. 
However, in 
order to classify degrees of freedom (DOFs) into subobjects, see, e.g., 
\cite{klawonn_dual-primal_2006,BadiaMartinNguyen_multi_material}, one 
needs to know the list of subsubdomains $\widehat{\Omega}_j$ each 
interface 
DOF belongs to. In a distributed-memory context, this can be obtained after a cheap nearest-neighbor communication 
sweep among original 
subdomains.
\end{remark}
}

Let $\widetilde{W}\subset W$ be the BDDC space (of BDDC-SO) 
consisting of functions in 
$W$ that have common values (continuous) at the selected coarse degrees 
of 
freedom. As we use a finer partition of the interface, the BDDC space 
associated with BDDC-SO is generally 
a subspace of the one associated with the 
standard BDDC. 

Next, let $E$ be the weighting operator (projection) from 
$\widetilde{W}$ onto $U$, that takes weighted average of contributions 
from subdomains, more precisely
\begin{equation}
\label{eq:W}
E u(\xi) \doteq  
\sum_{i:\,\Omega_i\ni \xi} \delta^{\dagger}_{i}(\xi)\,
u_{i}(\xi), \quad \hbox{where }  \delta^{\dagger}_{i}(\xi) \doteq 
\frac{\operatorname{card}(\{j:\; \widehat{\Omega}_j\ni \xi \text{ and } 
\widehat{\Omega}_j \subset \Omega_i\}) 
}{\operatorname{card}(\{j:\; \widehat{\Omega}_j\ni \xi\})}. 
\end{equation}
Here $\operatorname{card}(S)$ denotes the cardinality of the set $S$. 
\hnr{In general, the weight 
$\delta^{\dagger}_{i}(\xi)$ defined by \eqref{eq:W} is the same as the 
usual multiplicity weighting in, e.g., \cite{widlundToselli05}. However, 
for 
some 
(but 
not all) DOFs on the interface of 
the partition 
$\widehat{\Theta}$, the two weightings can be different. An example of such a DOF is 
depicted in Figure \ref{fig:objects} with a big solid dot and labeled 
as $\xi$. For this particular DOF, 
$$\delta^{\dagger}_{1}(\xi)=\frac{\operatorname{card(\{1,2\})}}{\operatorname{card(\{1,2,5\})}}=2/3,$$
while 
its multiplicity weight would be $1/2$.}

It can be checked that 
$a(\cdot,\cdot)=a^{\Theta}(\cdot,\cdot),\, V_1=U_I,\, V_2= \widetilde{W}, \, 
Q_1=I\text{ and } Q_2=(I-P)E$ satisfy the 
assumptions 
in Definition \ref{def:abstractBDDC}; see, e.g., \cite{mandel_multispace_2008}.

\section{Convergence Analysis}\label{sec:analysis}
In order to establish the condition number bound for the 
BDDC-SO preconditioner, we 
consider an auxiliary BDDC 
preconditioner, BDDC-A. This preconditioner is constructed similar to 
BDDC-SO using the abstract multispace framework introduced in Section 
\ref{sec_model}. However, there are three fundamental differences. 
First, BDDC-A is defined on the partition 
$\widehat{\Theta}=\{\widehat{\Omega}_j\}_{j=1}^{\widehat{N}}$ introduced 
in Section \ref{sec:formulation} to form subobjects. This partition 
has characteristic size $L$. The new partition leads to a larger space 
of discontinuous functions $\widehat{W}\subset W$ and a different 
projection $\widehat{P}$ onto the space of inner functions associated 
with $\widehat{\Theta}$. Second, 
coarse objects of BDDC-A are chosen 
such that it shares the same BDDC space $\widetilde{W}$ with BDDC-SO. This 
is done as follows: i) on $\Gamma$, subobjects of BDDC-SO, which are 
associated with regular geometrical objects (vertices, edges and faces 
of \hnr{subsubdomains} $\widehat{\Omega}_j$), are selected, ii) on  
$\widehat{\Gamma}\backslash\Gamma$, each degree of 
freedom is treated as a vertex. In other words, we impose full 
continuity on the part of the interface outside of $\Gamma$. Third, the 
weighting operator $\widehat{E}:\widetilde{W}\to U$ is the standard 
counting weighting operator defined by 
\begin{equation}
\label{eq:W2}
\widehat{E} u(\xi) \doteq  
\sum_{j:\,\widehat{\Omega}_j\ni \xi} \widehat{\delta}_{j}(\xi)\,
u_{j}(\xi), \quad \hbox{where }  \widehat{\delta}_{j}(\xi) \doteq 
\frac{1 
}{\operatorname{card}(\{j:\; \widehat{\Omega}_j\ni \xi\})}. 
\end{equation}
From \eqref{eq:W} and \eqref{eq:W2}, it should be noted that 
$\widehat{E}u=Eu$ for all $u\in\widetilde{W}$.

\hnr{Let 
$a^{\widehat{\Theta}}(\cdot,\cdot)$ be the extension of 
$a^{{\Theta}}(\cdot,\cdot)$
on $\widehat{W}$ associated with $\widehat{\Theta}$.} The following 
is 
our main result.     
\begin{theorem}\label{lem:1} 
The condition number $\kappa$ of the BDDC-SO preconditioner  is bounded 
by
\begin{equation}\label{eq:KBApb}
\kappa \leq \max\left\{\sup_{w\in\widetilde{W}} \frac{\|(I-P)E 
	w\|_{a^{\Theta}}}{\|w\|_{a^{\Theta}}}, 1 \right\}\leq 
	\max\left\{\sup_{w\in\widetilde{W}} 
	\frac{\|(\widehat{I}-\widehat{P})\widehat{E} 
	w\|_{a^{\widehat{\Theta}}}}{\|w\|_{a^{\widehat{\Theta}}}}, 1 
	\right\}\leq C 
	(1+\log(L/h))^2. 
\end{equation}
\end{theorem}
\begin{proof}
Since the formulation of BDDC-SO follows the abstract framework of 
multispace BDDC, its condition number $\kappa$ can be bounded as in the 
first inequality in \eqref{eq:KBApb}, see 
\cite{mandel_convergence_2003,mandel_multispace_2008}. In addition, as 
BDDC-A is a standard BDDC preconditioner with regular coarse 
objects, standard weighting operator, and defined on the partition 
$\widehat{\Theta}$ having characteristic size $L$, the third inequality 
in \eqref{eq:KBApb} is a standard result (see 
\cite{mandel_convergence_2003,mandel_multispace_2008}). We only need to 
prove the second inequality in \eqref{eq:KBApb}.

First, we note that $w\in\widetilde{W}$ can be discontinuous across 
$\Gamma$ but is always continuous in each $\Omega_i$. Therefore, 
$\|w\|_{a^{\Theta}}=\|w\|_{a^{\widehat{\Theta}}}$ for all
$w\in\widetilde{W}$. Second, as $P$ and $\widehat{P}$ are projections 
onto the space of inner functions that respectively vanish on $\Gamma$ 
and $\widehat{\Gamma}\supset\Gamma$, we have 
$(I-P)Ew|_{\Gamma}=Ew|_{\Gamma}=\widehat{E}w|_{\Gamma}=(\widehat{I}-\widehat{P})\widehat{E}w|_{\Gamma}$.
 On the other hand, $(I-P)Ew$ is discrete harmonic in $\Omega$ with 
 partition $\Theta$. This implies that it has the smallest 
 $a^{\Theta}(\cdot,\cdot)$-norm among functions in $W$ which share its 
 value on $\Gamma$. In other words,
 \[\|(I-P)E w\|_{a^{\Theta}} \leq \|(\widehat{I}-\widehat{P})\widehat{E} 
 w\|_{a^{{\Theta}}}= \|(\widehat{I}-\widehat{P})\widehat{E} 
 w\|_{a^{\widehat{\Theta}}}.\]
 Here, we note that the last equality is due to the fact that 
 $(\widehat{I}-\widehat{P})\widehat{E} w$ is continuous in each 
 $\Omega_i$. 
This finishes the 
proof.
\end{proof}
\begin{remark}\label{rmk1} 
	The BDDC-SO preconditioner can be extended 
	to the case with variable diffusion coefficients, i.e., when the 
	bilinear form becomes $a(u,v)=\int_{\Omega} \alpha \nabla u\cdot 
	\nabla 
	v\, dx$, where $\alpha$ varies from element to element. As our 
	analysis relies on the convergence of the auxiliary 
	preconditioner BDDC-A, the subobjects 
	need to be chosen in such a way that the coefficient has 
	to be constant in each \hnr{subsubdomain} $\widehat{\Omega}_j$. In 
	addition, the weighting operator $E$ in \eqref{eq:W} needs to be 
	updated to incorporate $\alpha$ with  $\widehat{\delta}_{i}(\xi) 
	\doteq {\left(\sum_{j:\, \widehat{\Omega}_j\ni \xi,\, 
	\widehat{\Omega}_j\subset \Omega_i}\alpha_j\right) 
	}\slash{\left(\sum_{j:\, \widehat{\Omega}_j\ni \xi} \alpha_j 
	\right)}.$
	We have explored this direction and 
	obtained first results for multi-material problems in 
	\cite{BadiaMartinNguyen_multi_material}. 
\end{remark}
\hnr{
\begin{remark}
	As the proof of Theorem \ref{lem:1} is based on the standard result for 
	an 
	auxiliary BDDC preconditioner, we need to fulfill all requirements 
	of that result \cite{klawonn.et.al.02,widlundToselli05}. Among these requirements, it is crucial to
	have enough primal constraints (the constraints we actually impose 
	on subobjects) so that for every pair of 
	neighboring subsubdomains, there exists an acceptable path, see 
	\cite{klawonn.et.al.02,widlundToselli05,klawonn_dual-primal_2006,BadiaMartinNguyen_multi_material}.
        In general, we can 
		guarantee the existence of acceptable paths by choosing either all edge 
		or all face subobjects on the original interface. Using all the possible
		vertex constraints also leads to a scalable preconditioner but is 
		less common in practice due to its poor performance, especially 
		in 
		3D, see, e.g.,\cite{widlundToselli05}. For algorithms 
		to find close-to-minimal set of primal constraints, we refer the 
		reader to 
		\cite{klawonn_dual-primal_2006,BadiaMartinNguyen_multi_material}.
\end{remark}}

\section{Numerical Experiments}\label{sec:numex}
The standard BDDC and BDDC-SO preconditioners have been implemented in 
\texttt{FEMPAR}, a parallel 
framework for the massively
parallel FE simulation of multiphysics problems 
governed by PDEs 
\cite{badia2018fempar,fempar}. We will test these preconditioners on 
MareNostrum 4 at Barcelona Supercomputing Center on the Poisson 
problem 
\eqref{weakform} with $f=1$, 
$\Omega$ the unit cube, and Dirichlet boundary conditions  
imposed on the whole boundary of $\Omega$.
\subsection{A preconditioner with $O(1)$ condition number}\label{num1}  
\begin{table}[htp!]
	\centering
	\begin{tabular}{lccccccrrrrrr}
		\toprule
		&& \multicolumn{5}{c}{ 
		number of iterations}&&\multicolumn{5}{c}{size of the coarse 
		problem 
		}\\ 
	    \cmidrule{3-7}\cmidrule{9-13}  
		Local problem size ($k$) && 4& 8&16&24 &32 
		&& 
		\multicolumn{1}{c}{4}&\multicolumn{1}{c}{8}&\multicolumn{1}{c}{16}&\multicolumn{1}{c}{24}
		 &\multicolumn{1}{c}{32} \\
		\hline
		BDDC(vef)   && 5&6  &9  &10 & 11 &&5,859&5,859&5,859&5,859&5,859 
		\\
		BDDC-SO(vef)   && 4&4 &4  &4 & 4 
		&&5,859&32,319&150,039&354,159&532,269\\
		BDDC-SO(vf-min)   && 5&4&4  &4  & 4 
		&&2,304&13,259&53,919&146,979&236,439
		\\
		\bottomrule
	\end{tabular}
	\caption{Number of PCG iterations required by BDDC and BDDC-SO 
	\hnr{and 
	associated size of the coarse problem} for 
		increasing local problem sizes $k$.}\label{Table1}
\end{table}
In the first 
experiment, which is a proof of concept of the result in Theorem \ref{lem:1}, 
we solve the problem on a set of uniform meshes 
with $10k  \times 10k\times 10k $ hexahedra, where $k=4, 8, 16, 24, 32$. 
We use $10^3$ processors, each of which is responsible for a local problem of 
size $k^3$ cells. In Table \ref{Table1}, we report the number of PCG 
iterations required to reduce the 2-norm of the residual by a factor of $10^6$ using 
BDDC and 
BDDC-SO preconditioners. For both of them, vertex, edge and face  
constraints (vef) will be used. For BDDC-SO, as the meshes get finer and 
finer 
we will use smaller subobjects to keep the ratio $L/h$ constant 
($L/h=4$). We can see that, even when all the possible constraints 
are used, the number of iterations of BDDC(vef) method increases as the 
local 
problem size increases. The number of iterations of BDDC-SO(vef), on 
the other hand, remains constant throughout. In addition, even when we 
drop edge constraints and only use half of the face constraints 
(vf-min), the number of iterations is still the same except for  
$k=4$. These results confirm the estimate in \eqref{eq:KBApb} and indicate 
that one does not need to impose constraints on all available subobjects 
to 
achieve the desired rate of convergence.    

\subsection{3D channels problem}
In practice, we rarely want to keep $L/h$ small at all cost as in 
Section \ref{num1} because the size of the 
coarse problem could be prohibitively large. In many situations, an introduction 
of a relatively small number of subobjects can greatly improve the  
convergence. 
In order to demonstrate that, we consider \eqref{weakform} with 
the modification 
described in Remark 
\ref{rmk1}. We assume that the material has the background with 
$\alpha=1$ and a system of 
channels with high diffusion coefficient $\alpha=10^{\ell},\, \ell=2, 4, 
6, 8$, see Figure \ref{fig:3d_channels} (left). First, we test the 
robustness of the two preconditioners when solving a small problem with 
the 
mesh  of $40\times 40 \times 40$ hexahedra partitioned into 
$4\times4\times4 $ subdomains. For BDDC-SO,  each subdomain 
is split into two smaller subdomains, one associated with the background 
and the other associated with the channels. These smaller subdomains 
form the fine partition $\{\widehat{\Omega}_j\}$ to define the 
subobjects. In Figure \ref{fig:3d_channels} (right), we plot the number 
of PCG iterations required by standard BDDC using all possible vertex, 
edge 
and face constraints (BDDC(vef)), and BDDC-SO using only face 
constraints (BDDC-SO(f)) to reduce the 2-norm of the residual by a factor of 
$10^{6}$. We can see that while the number of iterations of BDDC(vef) 
increases significantly as the contrast ($\ell$) increases, the number 
of iterations of BDDC-SO(f) is constant throughout. This verifies 
that BDDC-SO can be very robust for heterogeneous material problems. 
\begin{figure}[h!]
	\centering
	\includegraphics[width=6cm]{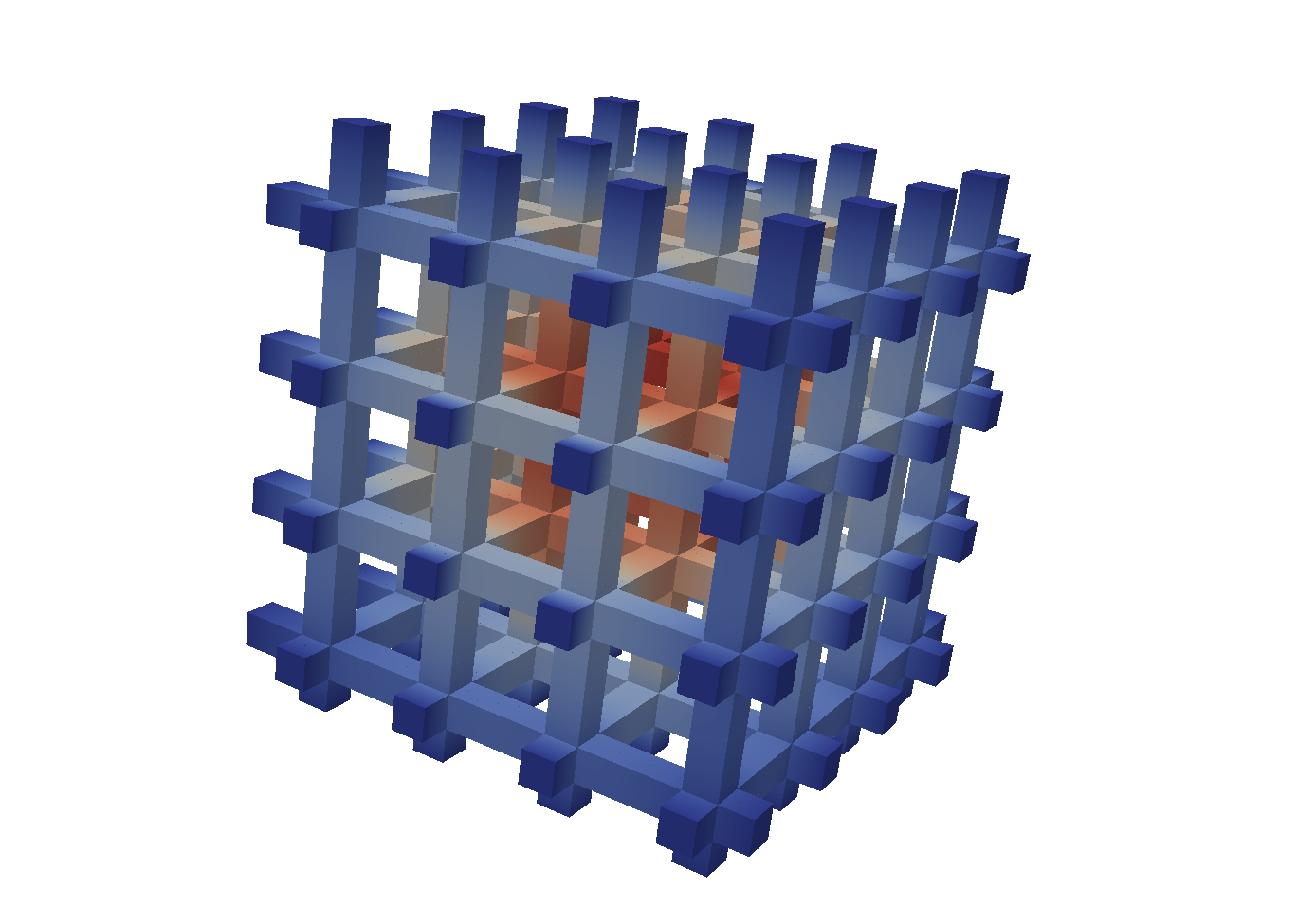}	\qquad
	\includegraphics[width=6cm]{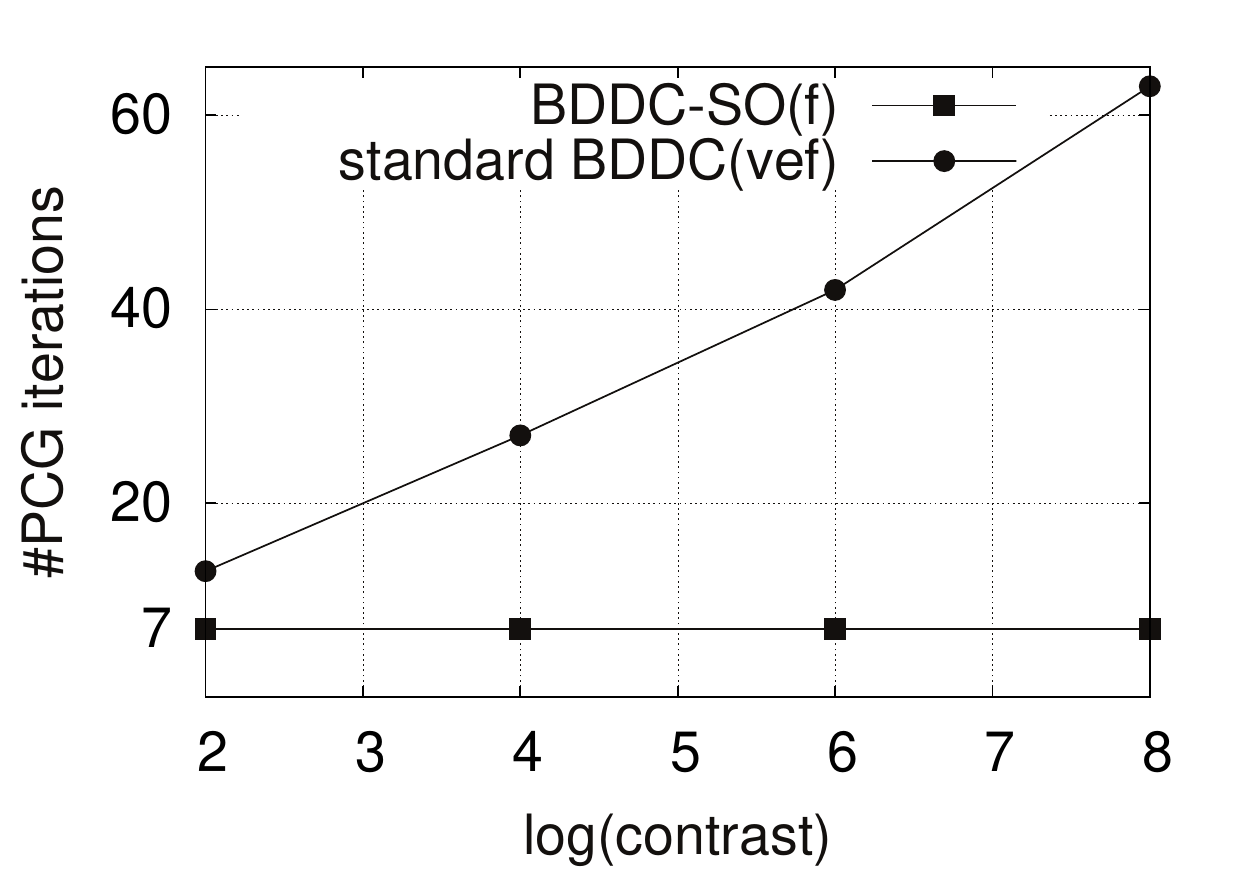}
	\caption{3D channels problem: channel of high 
		coefficient (left) and contrast robustness of BDDC and 
		BDDC-SO (right).}\label{fig:3d_channels}
\end{figure}  

Next, we consider $\ell=8$. We build 
uniform meshes with $32 m \times 32 m \times 32 m$ hexahedra. The 
largest mesh  has more than 260 million grid points. These meshes are 
partitioned into 
regular partitions of $m \times m \times m$ subdomains,
for $m=2,4,6,8,10$. The subobjects of BDDC-SO are defined similarly.
In Figure \ref{fig:3d_channels_iterations_time}, we present the weak 
scalability study in both PCG number of iterations and time (in seconds) 
of 
standard BDDC and BDDC-SO. For each preconditioner, we consider two 
variants: one using vertex and edge constraints (ve), the other using 
vertex, edge and face constraints (vef). We can observe
excellent weak scalability and robustness of BDDC-SO (bottom row).  On 
the other hand, the 
performance of standard BDDC (top row) varies widely as the number of 
processors 
increases. BDDC-SO methods are also much faster. They require less than 
21 iterations and less than 9 seconds to converge. Standard BDDC 
methods, on the 
other hand, need more than 1500 iterations and more than 300 seconds to 
converge when $m=6$ (512 subdomains).

\begin{figure}[htp!]
\centering	
\includegraphics[width=7.5cm]{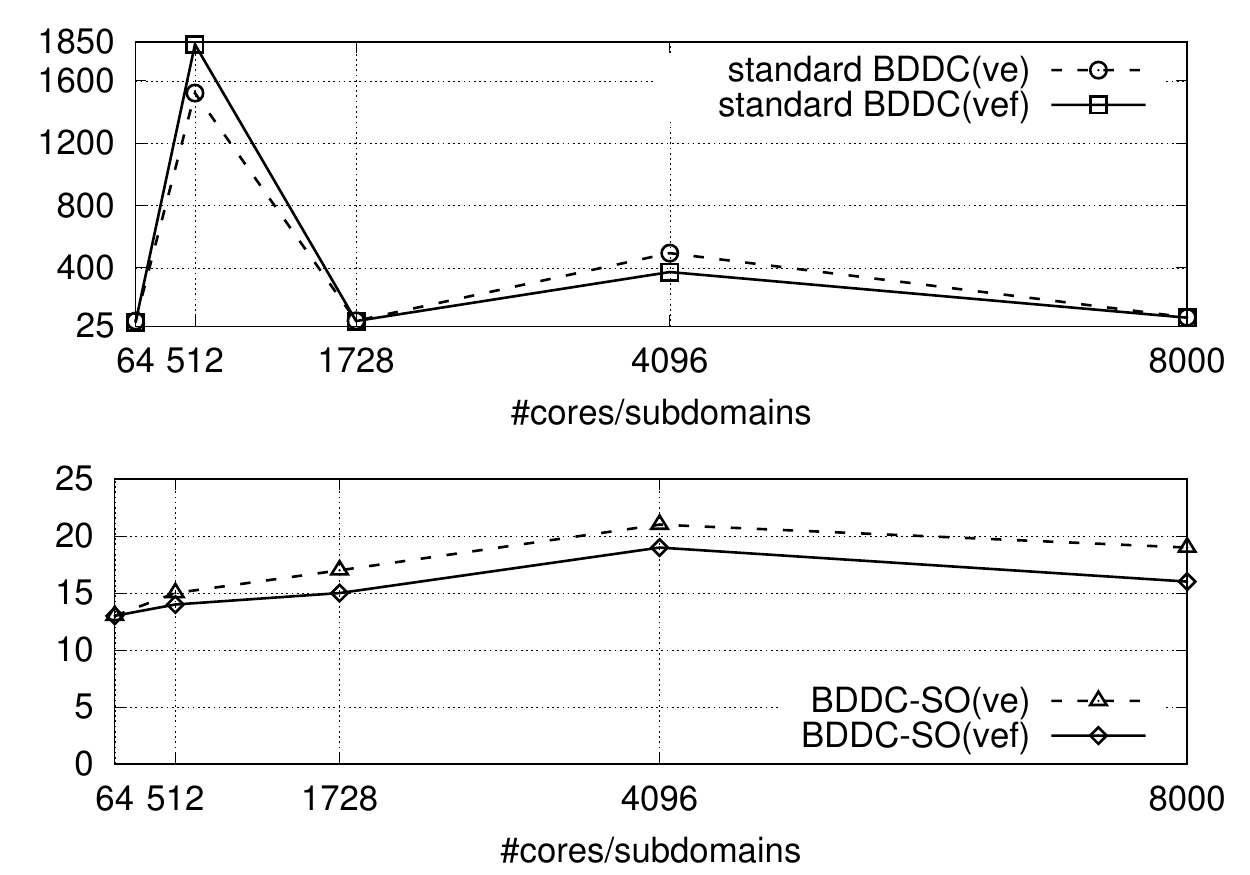}
		\qquad
\includegraphics[width=7.5cm]{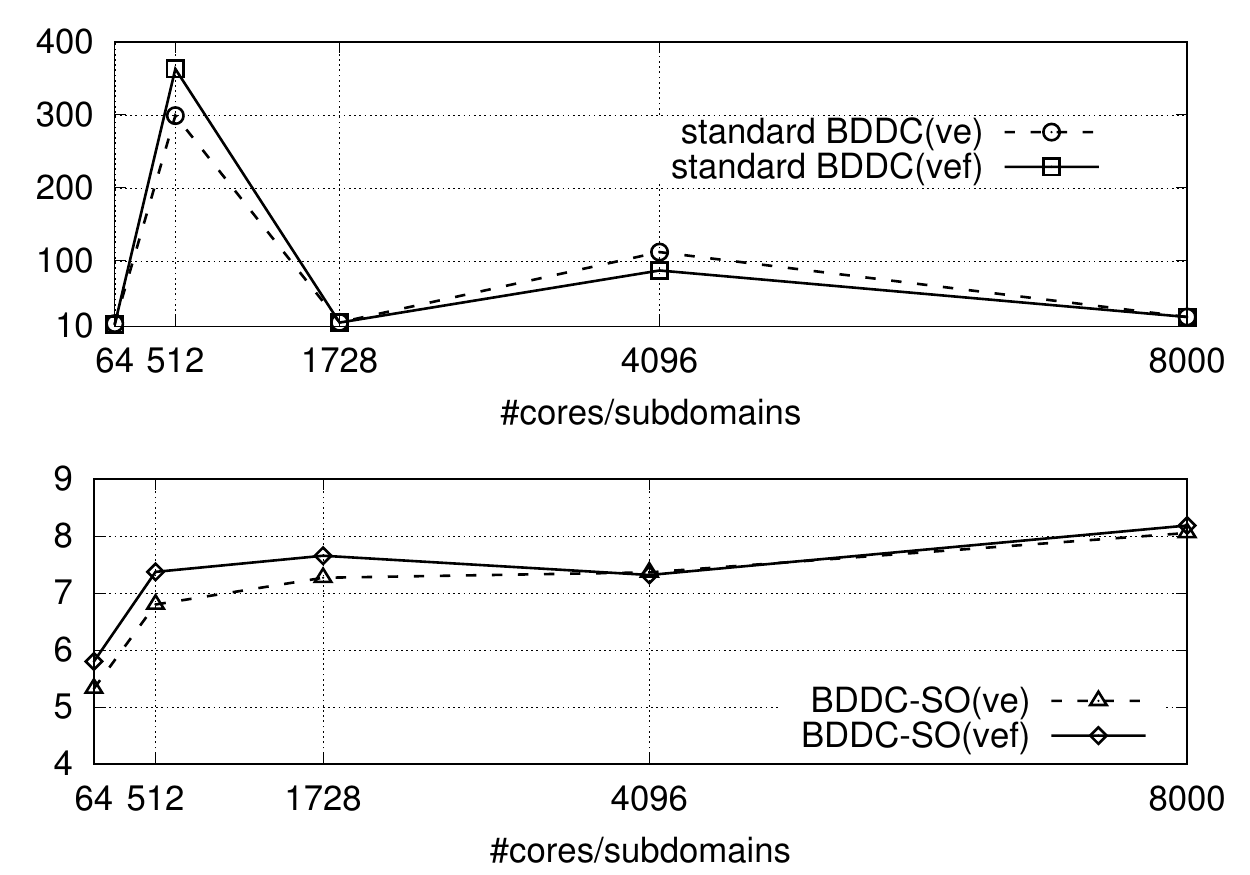}
	\caption{Weak scalability in  number of PCG 
		iterations (left) and in time in seconds (right).}
	\label{fig:3d_channels_iterations_time}
\end{figure}
\section{Conclusion}
We have introduced a new variant of BDDC preconditioner called BDDC-SO, 
where constraints are imposed on a selected set of subobjects: subedges, 
subfaces 
and vertices between subedges of  
subdomains. We are able to show that the new preconditioner has superior 
rate of convergence: the condition number depends only on the ratio of 
the characteristic sizes of the subobjects and the mesh. In 
addition, 
the proposed preconditioner is very robust w.r.t. the contrast in the 
physical coefficient. Numerical results verify these for problems with 
up to 260 millions of unknowns solved on 8K processors. 
However, the advantages come at the cost of solving  
larger coarse problems. Even though one does 
not need to impose constraints on all available subobjects to maintain 
the logarithmic bound of the condition number (only needs to have 
enough subobjects to guarantee the existence of the so-called {\it 
	acceptable paths} between neighboring fine subdomains, see 
\cite{klawonn.et.al.02,klawonn07_heterogenous,BadiaMartinNguyen_multi_material})
the size of the coarse problem of BDDC-SO can be significantly larger 
than that of BDDC. Therefore, two-level BDDC-SO is only suitable for 
small to medium size problems. In order to use BDDC-SO for large scale 
problems, one needs to extend BDDC-SO to its multi-level version, 
see \cite{tu_three-level_2007,mandel_multispace_2008} for the formulation 
of the multi-level BDDC method, and \cite{badia_extreme_2015} for its extreme scale HPC implementation. Indeed, 
BDDC-SO is very appealing for the multi-level implementation approach in 
\cite{badia_extreme_2015} \hnr{as users have great flexibility 
in defining subobjects at all coarser levels to have suitable coarser 
problem sizes and more balanced workload among 
levels.} 
 
\section*{Acknowledgments}
\fundthanks
\section*{References}
\bibliographystyle{plain}
\bibliography{bddc_short}


\end{document}

%% file: fig1.tex
\begin{figure}[htp!]
\begin{center}
\begin{tikzpicture}[scale=0.17]


\draw[line width=1](0,0) rectangle (20,20);
\draw [line width=1] (0,10) -- (20,10);
\draw [line width=1] (10,0) -- (10,20);

\node[above right, scale=1] at (1,1) {$\Omega_1$};
\node[above right, scale=1] at (1,11) {$\Omega_2$};
\node[above, scale=1] at (18,1) {$\Omega_3$};
\node[above, scale=1] at (18,11) {$\Omega_4$};


\node (A) at (10,0.35) {};
\node (B) at (10,9.5) {};

\node (C) at (10,10.5) {};
\node (D) at (10,19.6) {};

\node (E) at (0.35,10) {};
\node (F) at (9.5,10) {};

\node (G) at (10.5,10) {};
\node (H) at (19.65,10) {};

\draw [line width=3,opacity=0.6,blue,line cap=round,rounded corners] 
(A.center) -- (B.center);

\draw [line width=3,opacity=0.6,blue,line cap=round,rounded corners] 
(C.center) -- (D.center);

\draw [line width=3,opacity=0.6,blue,line cap=round,rounded corners] 
(E.center) -- (F.center);

\draw [line width=3,opacity=0.6,blue,line cap=round,rounded corners] 
(G.center) -- (H.center);

\foreach \x in {1,...,19}
\foreach \y in {1,...,19}
{
	\fill (\x,\y) circle (3pt);
}

\def\vertexmark{x};
\def\vertexsize{24pt};
\draw [black] plot [only marks, mark=\vertexmark,mark size=\vertexsize] 
coordinates 
{(10,10)};

\fill (10,5) circle (12pt);
\node[below left, scale=1.2] at (10,5.5) {$\xi$};

\end{tikzpicture}
\hspace*{1cm}	
\begin{tikzpicture}[scale=0.17]


\draw[line width=1](0,0) rectangle (20,20);
\draw [line width=1] (0,10) -- (20,10);
\draw [line width=1] (10,0) -- (10,20);

\node[above right, scale=1] at (1,1.3) {$\widehat{\Omega}_1$};
\node[above right, scale=1] at (6,5.2) {$\widehat{\Omega}_2$};
\node[above right, scale=1] at (1,11) {$\widehat{\Omega}_3$};
\node[above right, scale=1] at (6,11) {$\widehat{\Omega}_4$};
\node[above, scale=1] at (13,1) {$\widehat{\Omega}_5$};
\node[above, scale=1] at (17.8,1) {$\widehat{\Omega}_6$};

\node[above, scale=1] at (13,13) {$\widehat{\Omega}_7$};
\node[above, scale=1] at (17.8,13) {$\widehat{\Omega}_8$};


\node (A) at (10,0.25) {};
\node (B) at (10,9.6) {};
\node (BB) at (10,4.7) {};
\node (AA) at (10,5.3) {};

\node (C) at (10,10.4) {};
\node (D) at (10,19.8) {};
\node (CC) at (10,16.3) {};
\node (DD) at (10,15.7) {};

\node (E) at (0.25,10) {};
\node (F) at (9.65,10) {};
\node (EE) at (5.3,10) {};
\node (FF) at (4.7,10) {};

\node (G) at (10.4,10) {};
\node (H) at (19.7,10) {};
\node (HH) at (14.7,10) {};
\node (GG) at (15.3,10) {};

\draw [line width=1,line cap=round] 
(10,5)--(5,5) -- (5,10) -- (5,16) 
--(10,16);

\draw [line width=1,line cap=round] 
(15,0)--(15,20);

\draw [line width=3,opacity=0.6,blue,line cap=round,rounded corners] 
(A.center) -- (BB.center);
\draw [line width=3,opacity=0.6,blue,line cap=round,rounded corners] 
(AA.center) -- (B.center);

\draw [line width=3,opacity=0.6,blue,line cap=round,rounded corners] 
(C.center) -- (DD.center);
\draw [line width=3,opacity=0.6,blue,line cap=round,rounded corners] 
(CC.center) -- (D.center);

\draw [line width=3,opacity=0.6,blue,line cap=round,rounded corners] 
(E.center) -- (FF.center);
\draw [line width=3,opacity=0.6,blue,line cap=round,rounded corners] 
(EE.center) -- (F.center);

\draw [line width=3,opacity=0.6,blue,line cap=round,rounded corners] 
(G.center) -- (HH.center);
\draw [line width=3,opacity=0.6,blue,line cap=round,rounded corners] 
(GG.center) -- (H.center);

\foreach \x in {1,...,19}
    \foreach \y in {1,...,19}
    {
       \fill (\x,\y) circle (3pt);
    }
    
\def\vertexmark{x};
\def\vertexsize{24pt};
\draw [black] plot [only marks, mark=\vertexmark,mark size=\vertexsize] 
coordinates 
{(10,10) (10,5) (10,16) (5,10) (15,10)};
\fill (10,5) circle (12pt);
\node[below left, scale=1.2] at (10,5.5) {$\xi$};
    
\end{tikzpicture}
\end{center}
\caption{Standard objects (left) and subobjects (right): vertices and 
(sub)edges are 
marked by crosses and bold shade.}\label{fig:objects}
\end{figure}

%% file: badia_martin_nguyen_bddc_subobjects.bbl
\begin{thebibliography}{10}

\bibitem{BadiaMartinNguyen_multi_material}
Santiago Badia, Alberto~F. Mart\'in, and Hieu Nguyen.
\newblock Physics-based balancing domain decomposition by constraints for
  multi-material problems, preprint
  https://hal.archives-ouvertes.fr/hal-01337968v4, 2018.

\bibitem{badia_extreme_2015}
Santiago Badia, Alberto~F. Mart{\'{\i}}n, and Javier Principe.
\newblock Multilevel balancing domain decomposition at extreme scales.
\newblock {\em SIAM J. Sci. Comput.}, 38(1):C22--C52, 2016.

\bibitem{fempar}
Santiago Badia, Alberto~F. Mart\'in, and Javier Principe.
\newblock \texttt{FEMPAR} {W}eb page.
\newblock \url{http://www.fempar.org}, 2017.

\bibitem{badia2018fempar}
Santiago Badia, Alberto~F Mart{\'\i}n, and Javier Principe.
\newblock \texttt{FEMPAR}: An object-oriented parallel finite element
  framework.
\newblock {\em Archives of Computational Methods in Engineering},
  25(2):195--271, 2018.

\bibitem{BadiaNguyen_perturbed}
Santiago Badia and Hieu Nguyen.
\newblock Balancing domain decomposition by constraints and perturbation.
\newblock {\em SIAM Journal on Numerical Analysis}, 54(6):3436--3464, 2016.

\bibitem{BadiaNguyenDD23}
Santiago Badia and Hieu Nguyen.
\newblock Relaxing the roles of corners in bddc by perturbed formulation.
\newblock In Chang-Ock Lee, Xiao-Chuan Cai, David~E. Keyes, Hyea~Hyun Kim, Axel
  Klawonn, Eun-Jae Park, and Olof~B. Widlund, editors, {\em Domain
  Decomposition Methods in Science and Engineering XXIII}, pages 397--405,
  Cham, 2017. Springer International Publishing.

\bibitem{BrennerScott08}
Susanne~C. Brenner and L.~Ridgway Scott.
\newblock {\em The mathematical theory of finite element methods}, volume~15 of
  {\em Texts in Applied Mathematics}.
\newblock Springer, New York, third edition, 2008.

\bibitem{brenner_bddc_2007}
Susanne~C. Brenner and Li-Yeng Sung.
\newblock {BDDC} and {FETI}-{DP} without matrices or vectors.
\newblock {\em Computer Methods in Applied Mechanics and Engineering},
  196(8):1429--1435, January 2007.

\bibitem{dohrmann_preconditioner_2003}
Clark~R. Dohrmann.
\newblock A {Preconditioner} for {Substructuring} {Based} on {Constrained}
  {Energy} {Minimization}.
\newblock {\em SIAM Journal on Scientific Computing}, 25(1):246--258, 2003.

\bibitem{klawonn07_heterogenous}
Axel Klawonn and Oliver Rheinbach.
\newblock Robust {FETI}-{DP} methods for heterogeneous three dimensional
  elasticity problems.
\newblock {\em Comput. Methods Appl. Mech. Engrg.}, 196(8):1400--1414, 2007.

\bibitem{klawonn_dual-primal_2006}
Axel Klawonn and Olof~B. Widlund.
\newblock Dual-primal {FETI} methods for linear elasticity.
\newblock {\em Communications on Pure and Applied Mathematics},
  59(11):1523--1572, 2006.

\bibitem{klawonn.et.al.02}
Axel Klawonn, Olof~B. Widlund, and Maksymilian Dryja.
\newblock Dual-primal {FETI} methods for three-dimensional elliptic problems
  with heterogeneous coefficients.
\newblock {\em SIAM J. Numer. Anal.}, 40(1):159--179 (electronic), 2002.

\bibitem{mandel_convergence_2003}
Jan Mandel and Clark~R. Dohrmann.
\newblock Convergence of a balancing domain decomposition by constraints and
  energy minimization.
\newblock {\em Numerical Linear Algebra with Applications}, 10(7):639--659,
  2003.

\bibitem{mandel_algebraic_2005}
Jan Mandel, Clark~R. Dohrmann, and Radek Tezaur.
\newblock An algebraic theory for primal and dual substructuring methods by
  constraints.
\newblock {\em Applied Numerical Mathematics}, 54(2):167--193, July 2005.

\bibitem{mandel_multispace_2008}
Jan Mandel, Bed{\v r}ich Soused{\'\i}k, and Clark Dohrmann.
\newblock Multispace and multilevel {BDDC}.
\newblock {\em Computing}, 83(2):55--85, 2008.

\bibitem{widlundToselli05}
Andrea Toselli and Olof Widlund.
\newblock {\em Domain decomposition methods---algorithms and theory}, volume~34
  of {\em Springer Series in Computational Mathematics}.
\newblock Springer-Verlag, Berlin, 2005.

\bibitem{tu_three-level_2007}
Xuemin Tu.
\newblock Three-{Level} {BDDC} in {Three} {Dimensions}.
\newblock {\em SIAM Journal on Scientific Computing}, 29(4):1759--1780, January
  2007.

\end{thebibliography}
